\documentclass[reqno,b5paper]{amsart}

\usepackage[T2A]{fontenc}
\usepackage[english]{babel}
\usepackage{graphicx}
\usepackage{amsmath}
\usepackage{amsfonts}
\usepackage{amssymb}
\usepackage{amsthm}

\usepackage{enumerate}
\usepackage[mathscr]{eucal}

\theoremstyle{plain}
\newtheorem{theorem}{Theorem}

\newtheorem{prop}[theorem]{Proposition}
\newtheorem{lemma}[theorem]{Lemma}

\theoremstyle{remark}

\theoremstyle{definition}

\DeclareMathOperator{\pr}{pr} \DeclareMathOperator{\rp}{rp\,}
\DeclareMathOperator{\sa}{sa} \DeclareMathOperator{\un}{un}
\DeclareMathOperator{\RE}{Re\,} \DeclareMathOperator{\tr}{Tr\,}
\DeclareMathOperator{\rg}{rg\,}
 \DeclareMathOperator{\im}{i\,}

\begin{document}

\title[Orthogonal vector fields over  $W^*$-algebras of type $I_2$]{A characterization of orthogonal
vector fields over  ${\bf W^*}$-algebras of type ${\bf I_2}$}

\author[G.D. Lugovaya and A.N. Sherstnev]{G.D. Lugovaya and A.N. Sherstnev*}

\newcommand{\acr}{\newline\indent}

\address{\llap{*} Department of Mathematical Analysis \acr
Kazan Federal University \acr Kremlyovskaya St., 18 \acr Kazan
420008 \acr RUSSIAN FEDERATION}

\email{glugovay@kpfu.ru}

\email{Anatolij.Sherstnev@ksu.ru}



\begin{abstract}
In the paper we give a characterization of  a $w^*$-continuous orthogonal
vector field $F$ over an $W^*$-algebra $\mathcal{N}$ of type  $I_2$ in terms of  reductions $F$ on the center
of $\mathcal{N}$. As an application it is obtained a proof of the assertion that an arbitrary
$w^*$-continuous orthogonal vector field over a $W^*$-algebra  of type  $I_2$ is stationary.
\end{abstract}
\maketitle

\bigskip
\centerline{\textbf{Introduction}}

\bigskip In [1],[2] Masani studied the integration with respect to orthogonal vector measures defined
on rings of the sets. These works obtained an extension in numerous publications in context of
the noncommutative measure theory. Some settings of problems and approaches to their solutions may be
found in [3], [4,\S 31]. One of intrinsic problem in this subject is the problem of characterization of
Hilbert-valued linear mappings of $W^*$-algebra preserving the orthogonality property (so-called orthogonal
vector fields).  In some sense there is a standard (well known) result for the commutative $W^*$-algebra
(see for instance [4, Theorem 31.19]).

\medskip In the paper we give a characterization of  a $w^*$-continuous orthogonal
vector field $F$ over an $W^*$-algebra $\mathcal{N}$ of type  $I_2$ in terms of  reductions $F$ on the center
of $\mathcal{N}$. As an application it is obtained a proof of the assertion that an arbitrary
$w^*$-continuous orthogonal vector field over a $W^*$-algebra  of type  $I_2$ is stationary.

\bigskip
\centerline{\textbf{1. Preliminaries}}

\bigskip  Let $\mathcal{A}$ be a $W^*$-algebra, and $\mathcal{A}^{\pr},\
 \mathcal{A}^{\sa},\ \mathcal{A}^{\un}$
denote the sets of orthogonal
projections, selfadjoint, unitaries elements in $\mathcal{A}$, respectively.
We  denote by $\rp(x)$ the range projection of $x\in \mathcal{A}^+$. It is the least projection
of all projections
$p\in\mathcal{A}^{\pr}$ such that $px=x$. In the paper  we shall mainly deal with $W^*$-algebras  of type  $I_2$.
It is  known that any $W^*$-algebra of type  $I_2$ can be represented in the form  $\mathcal{N} =\mathcal{M}\otimes M_2$
where  $\mathcal{M}$ is a commutative $W^*$-algebra and $M_2$ is the algebra of all $2\times2$ matrices over
$\mathbb{C}$. So the elements of $\mathcal{N}$ are matrices in the form of $(x_{ij})=
\left(\begin{array}{cc} x_{11} & x_{12}\\
x_{21} & x_{22} \end{array}\right)$ where $x_{ij}\in\mathcal{M}$. We denote by $\textbf{I}$ and $\textbf{1}$
the units of algebras $\mathcal{N}$ and  $\mathcal{M}$ respectively; $\{\varepsilon_{ij}\}_{i,j=1,2}$ denote
the matrix unit of algebra $\mathcal{N}$ in the form
$$
\varepsilon_{11}= \left(\begin{array}{cc} \textbf{1} & 0\\
0 & 0 \end{array}\right),\ \varepsilon_{12}= \left(\begin{array}{cc} 0 & \textbf{1}\\
0 & 0 \end{array}\right),\ \varepsilon_{21}= \left(\begin{array}{cc} 0 & 0\\
\textbf{1} & 0 \end{array}\right),\ \varepsilon_{22}= \left(\begin{array}{cc} 0 & 0\\
0 & \textbf{1} \end{array}\right).$$

It is known [5, Prop. 1.18.1] that a commutative  $W^*$-algebra $\mathcal{M}$ may be
realized as algebra  $L^{\infty}(\Omega,\nu)$ of all essentially bounded locally $\nu$-measurable
functions on a localizable measure space ($\Omega, \nu)$ (i. e.  $\Omega$ is direct sum of finite measure
spaces, see [6]). In this case, the Banach space $L^1(\Omega,\nu)$ is the predual of
$L^{\infty}(\Omega,\nu)$.  Now  we shall identify  $\mathcal{M}$ with $L^{\infty}(\Omega,\nu)$.
In this case the characteristic functions
$$
\pi(\omega)\equiv\chi_{\pi}(\omega)=\left\{\begin{array}{ll}
1, &\text{if } \omega\in\pi,\\
0, & \text{if }\omega\not\in\pi,\end{array}\right.\qquad \pi\subset\Omega. $$
correspond to projections $\pi\in\mathcal{M}^{\pr}$. (We use the same letter $\pi$ to designate three objects:
a projection in  $\mathcal{M}^{\pr}$, a $\nu$-measurable set in $\Omega$ and  the characteristic function
of this set.)

Accordingly, the  $W^*$-algebra $\mathcal{N}=\mathcal{M}\otimes M_2$ is realized as the algebra of
$2\times 2$-matrices  $(x_{ij}),\ x_{ij}\in L^{\infty}(\Omega,\nu)$ with the predual
$$
\mathcal{N}_*=L^1(\Omega,\nu)\otimes(M_2)_*=L^1(\Omega,\nu, M_2)\eqno(1)$$
where $L^1(\Omega,\nu, M_2)$ is the Banach space of all $M_2$-valued Bochner $\nu$-integrable functions on
$\Omega$ [5, Prop. 1.22.12]. Elements of the space $M_2$ on the right-hand side of (1) can be viewed as density matrices.
In other words, elements $\varphi$ of $\mathcal{N}_*$ are the matrices
$$
\varphi=\left(\begin{array}{cc} \varphi_{11} & \varphi_{12}\\
\varphi_{21} & \varphi_{22} \end{array}\right),\quad \varphi_{ij}\in L^1(\Omega,\nu),\eqno(2)$$
and the action of $\varphi$ on an element $a=(a_{ij}),\ a_{ij}\in L^{\infty}(\Omega,\nu)$, is given
by the equality
$$
\varphi(a)=\int\limits_{\Omega}\tr[(\varphi_{ij})(a_{ij})]\,d\nu=\int\limits_{\Omega}(\varphi_{11}a_{11}
+\varphi_{12}a_{21}+\varphi_{21}a_{12}+\varphi_{22}a_{22})\,d\nu.$$
In this case
$$
\varphi(\pi\varepsilon_{ij})=\int\limits_{\pi}\varphi_{ji}\,d\nu,\quad\pi\in\mathcal{M}^{\pr},\ i,j=1,2,
\footnote{ We use the notations of type $a\varepsilon_{12}\equiv\left(\begin{array}{cc} 0 & a\\
0 & 0 \end{array}\right),\ a\in\mathcal{M}$.}$$
i. e.  $\varphi_{ij}=\frac{d}{d\nu}\varphi((\cdot)\varepsilon_{ji})$ is the Radon-Nikodym derivative
of the charge $\pi\rightarrow\varphi(\pi\varepsilon_{ji}),\ \pi\in\mathcal{M}^{\pr}$, with respect
to measure $\nu$.

\medskip   Let  $H$ be a complex Hilbert space. A bounded linear map
$F:\mathcal{A}\to H$ is said to be  \textit{an orthogonal vector field (OVF)} if
$$
pq=0\ (p,q\in\mathcal{A}^{\pr})\Rightarrow \langle F(p),F(q)\rangle =0. \eqno(3)$$
If a linear map $F$ with the property (3) is continuous in  $w^*$-topology on $\mathcal{A}$
and weak topology on $H$, then the map $F$ is referred to as \textit{$w^*$-continuous OVF}.
It  should be noted that a $w^*$-continuous OVF $F$ is  bounded, so that it is an OVF
[7, Corollary 3]. A $w^*$-continuous  OVF $F$ is said to be \textit{stationary}  [3], if there exist two
functionals $\varphi,\psi\in \mathcal{A}_*^+$ such that
$$
\langle F(x),F(y)\rangle= \varphi(y^*x)+\psi(xy^*),\quad x,y\in\mathcal{A}.\eqno(4)$$
Given OVF $F:\mathcal{A}\to H$ assign a positive linear functional $\varrho\in\mathcal{A}^*$,
$$
\varrho(x)\equiv\langle F(x),F(1)\rangle,\quad x\in\mathcal{A}\eqno(5)$$
($\varrho\in\mathcal{A}_*^+$ as soon as $F$ is $w^*$-continuous). The following equalities hold
$$
\|F(x)\|^2=\varrho(x^2),\quad x\in\mathcal{A}^{\sa}, \eqno(6)$$
$$
\RE\langle F(x),F(y)\rangle=\frac{1}{2}\varrho(xy+yx),\quad x,y\in \mathcal{A}^{\sa}.\eqno(7)$$
Equality (6) follows from the spectral theorem with regard to (3), and the following computation gives (7):
\begin{align}
\RE\langle F(x),F(y)\rangle&=\frac{1}{4}[\langle F(x+y),F(x+y)- F(x-y),F(x-y)\rangle]\nonumber\\
&=\frac{1}{4}[\varrho((x+y)^2)-\varrho((x-y)^2)]=\frac{1}{2}\varrho(xy+yx),\quad x,y\in \mathcal{A}^{\sa}.
\nonumber\end{align}
It is easily seen that for a stationary OVF the next equality is valid
$$
\varrho=\varphi+\psi.\eqno(8)$$

\bigskip
\centerline{\textbf{2. Orthogonal vector fields over ${\bf W^*}$-algebras of type ${\bf I_2}$}}

\bigskip Let $\mathcal{N}=\mathcal{M}\otimes M_2$ be a $W^*$-algebra of type $I_2$, $H$
be a Hilbert space over $\mathbb{C}$ and $F:\mathcal{N}\to H$ be a
 $H$-valued OVF. By the proof of Theorem 31.6 [4], the mappings $F_{ij}:\mathcal{M}\to H$
 given by the equalities
$$
F_{ij}(a)\equiv F(a\varepsilon_{ij}),\quad a\in\mathcal{M},\ i,j=1,2,\eqno(9)$$
are OVFs over $\mathcal{M}$. Moreover,
$$
F(x)=\sum_{i,j=1}^{2}F_{ij}(x_{ij}),\quad x=(x_{ij})\in\mathcal{N}.$$
Since the center of $\mathcal{N}$ is isomorphic to algebra  $\mathcal{M}$, we can consider
the orthogonal vector fields $F_{ij}$ as reductions $F$ on the center of  $\mathcal{N}$.

\begin{prop}
Let $F:\mathcal{N}\to H$ be a OVF and $F_{ij}$ are defined in  (9).
Then
\begin{itemize}
\item[(i)]\ $\langle F_{11}(a),F_{22}(b)\rangle=\langle F_{12}(a),F_{21}(b)\rangle=0,\quad a,b\in\mathcal{M}$,
\item[(ii)]\ $\|F_{12}(a)\|^2 + \|F_{21}(a)\|^2=\|F_{11}(a)\|^2+\|F_{22}(a)\|^2,\quad a\in\mathcal{M}$,
\item[(iii)] $\langle F_{ij}(a),F_{kl}(b)\rangle=\langle F_{ij}(b^*a),F_{kl}(\textbf{1})\rangle,\quad
a,b\in\mathcal{M},\ i,j,k,l=1,2$,
\item[(iv)] $\langle F_{12}(\pi),F_{11}(\textbf{1})\rangle=\langle F_{22}(\textbf{1}),F_{21}(\pi)\rangle,\
\langle F_{21}(\pi),F_{11}(\textbf{1})\rangle= \langle F_{22}(\textbf{1}),F_{12}(\pi)\rangle,\linebreak
\pi\in\mathcal{M}^{\pr}$.\end{itemize}
\end{prop}

\begin{proof}
 Let us verify the second equality in (i). In view of the spectral theorem
 it suffices to prove the assertion in case  $a=\sigma,\ b=\tau$ are projections in  $\mathcal{M}^{\pr}$.
Let us first assume that  $\sigma\tau=0$. We have
$$
\left(\begin{array}{cc} 0 & \sigma\\
0 & 0 \end{array}\right) = \frac{1}{4}\sum_{\omega=\pm1,\pm\im}^{}\omega\left(\begin{array}{cc}
 \sigma & \overline{\omega}\sigma\\
\omega\sigma & \sigma \end{array}\right),\ \left(\begin{array}{cc} 0 & 0\\
\tau & 0 \end{array}\right)= \frac{1}{4}\sum_{\omega}^{}\omega\left(\begin{array}{cc} \tau &\omega\tau\\
 \overline{\omega}\tau & \tau \end{array}\right),\eqno(10)$$
where in right-hand sides of (10) there are linear combinations of projections from $\mathcal{N}$.
It now follows in view of (3) that  $\sigma\tau=0\Rightarrow \langle F_{12}(\sigma),F_{21}(\tau)\rangle=0$.

Now we suppose that $\sigma=\tau$. By the equalities
$$
\left(\begin{array}{cc} 0 & \sigma\\
0 & 0 \end{array}\right) = \frac{1}{2}\left(\begin{array}{cc} 0 & \sigma\\
\sigma & 0 \end{array}\right)+\frac{\im}{2}\left(\begin{array}{cc} 0 & -\im\sigma\\
\im\sigma & 0 \end{array}\right),$$
$$ \left(\begin{array}{cc} 0 & 0\\
\sigma & 0 \end{array}\right)= \frac{1}{2}\left(\begin{array}{cc} 0 & \sigma\\
\sigma & 0 \end{array}\right)+\frac{\im}{2}\left(\begin{array}{cc} 0 & \im\sigma\\
-\im\sigma & 0 \end{array}\right)$$
we have
\begin{align}
4&\langle F_{12}(\sigma),F_{21}(\sigma)\rangle =4\left\langle F\left(\begin{array}{cc} 0 & \sigma\\
0 & 0 \end{array}\right),F\left(\begin{array}{cc} 0 & 0\\
\sigma & 0 \end{array}\right)\right\rangle\nonumber\\
&=\left\|F\left(\begin{array}{cc} 0 & \sigma\\
\sigma & 0 \end{array}\right)\right\|^2
+\left\langle F\left(\begin{array}{cc} 0 & -\im\sigma\\
\im\sigma & 0 \end{array}\right),F\left(\begin{array}{cc} 0 & \im\sigma\\
-\im\sigma & 0 \end{array}\right)\right\rangle\nonumber\\
& -i\left\langle F\left(\begin{array}{cc} 0 & \sigma\\
\sigma & 0 \end{array}\right),F\left(\begin{array}{cc} 0 & \im\sigma\\
-\im\sigma & 0 \end{array}\right)\right\rangle+\im\left\langle F\left(\begin{array}{cc} 0 & -\im\sigma\\
\im\sigma & 0 \end{array}\right),F\left(\begin{array}{cc} 0 & \sigma\\
\sigma & 0 \end{array}\right)\right\rangle \nonumber\\
&=\left\|F\left(\begin{array}{cc} 0 & \sigma\\
\sigma & 0 \end{array}\right)\right\|^2-\left\|F\left(\begin{array}{cc} 0 & -\im\sigma\\
\im\sigma & 0 \end{array}\right)\right\|^2\nonumber\\
&+\im\left\{\left\langle F\left(\begin{array}{cc} 0 & \sigma\\
\sigma & 0 \end{array}\right),F\left(\begin{array}{cc} 0 & -\im\sigma\\
\im\sigma & 0 \end{array}\right)\right\rangle+\left\langle F\left(\begin{array}{cc} 0 & -\im\sigma\\
\im\sigma & 0 \end{array}\right),F\left(\begin{array}{cc} 0 & \sigma\\
\sigma & 0 \end{array}\right)\right\rangle \right\}.\nonumber\end{align}
In addition (see (6),(7)),
\begin{align}
\left\|F\left(\begin{array}{cc} 0 & -\im\sigma\\
\im\sigma & 0 \end{array}\right)\right\|^2&=\left\langle F\left(\left(\begin{array}{cc} 0 & -\im\sigma\\
\im\sigma & 0 \end{array}\right)^2\right), F(\textbf{I})\right\rangle\nonumber\\
&=\left\langle F\left(\begin{array}{cc} \sigma & 0\\
0 & \sigma \end{array}\right), F(\textbf{I})\right\rangle
=\left\|F\left(\begin{array}{cc} 0 & \sigma\\
\sigma & 0 \end{array}\right)\right\|^2,\nonumber\end{align}
The expression in braces vanishes:
\begin{align}
\left\{...\right\}&=2\RE\left\langle F\left(\begin{array}{cc} 0 & \sigma\\
\sigma & 0 \end{array}\right),F\left(\begin{array}{cc} 0 & -\im\sigma\\
\im\sigma & 0 \end{array}\right)\right\rangle\nonumber\\
&=\varrho\left[\left(\begin{array}{cc} 0 & \sigma\\
\sigma & 0 \end{array}\right)\left(\begin{array}{cc} 0 & -\im\sigma\\
\im\sigma & 0 \end{array}\right)+\left(\begin{array}{cc} 0 & -\im\sigma\\
\im\sigma & 0 \end{array}\right)\left(\begin{array}{cc} 0 & \sigma\\
\sigma & 0 \end{array}\right)\right]=0,\nonumber\end{align}
and the second equality in (i) is  established. The equality
$\langle F_{11}(a),F_{22}(b)\rangle=0\ (a,b\in\mathcal{M})$ is obvious in view of the spectral theorem.

\medskip Equality  (ii) follows from computation (with regard to (i))
\begin{align}
\|F_{12}(a)\|^2&+ \|F_{21}(a)\|^2=\|F_{12}(a)\|^2+ \|F_{21}(a^*)\|^2=\|F_{12}(a)+ F_{21}(a^*)\|^2\nonumber\\
&=\left\langle F\left(\begin{array}{cc} 0 & a\\
a^* & 0 \end{array}\right),F\left(\begin{array}{cc} 0 & a\\
a^* & 0 \end{array}\right)\right\rangle
= \left\langle F\left(\left(\begin{array}{cc} 0 & a\\
a^* & 0 \end{array}\right)^2\right),F(\textbf{I})\right\rangle\nonumber\\
&=\langle F_{11}(a^*a),F_{11}(\textbf{1})\rangle+\langle F_{22}(a^*a),F_{22}(\textbf{1})\rangle
=\|F_{11}(a)\|^2+ \|F_{22}(a)\|^2.\nonumber\end{align}
We used here (for $G=F_{21}$) the following property of the OVF over commutative
 $W^*$-algebra $\mathcal{M}$ (it follows easily from the spectral theorem):
$$
\|G(a)\|^2=\|G(a^*)\|^2, \quad a\in\mathcal{M}.$$
Let us establish (iii).  We first note that representations (10) give
$$
\sigma\tau=0\ (\sigma,\tau\in\mathcal{M}^{\pr})\Rightarrow \langle F_{ii}(\sigma),F_{12}(\tau)\rangle=
 \langle F_{ii}(\sigma),F_{21}(\tau)\rangle
=0,\quad i=1,2.\eqno(11)$$
In view of the  spectral theorem and boundedness of linear maps $F_{ij}$ it is sufficient
to establish (iii) for elements  $a,b$ in the form
$$
a=\sum_{t}\mu_t\pi_t,\ b=\sum_{t}\lambda_t\pi_t\quad(\lambda_t,\mu_t\in \mathbb{C},\ \pi_t
\in\mathcal{M}^{\pr},\ \pi_t\pi_s=0\ (t\ne s)).$$
(Here the sums are finite and  $\sum\limits_t\pi_t=\textbf{1}$.) Taking into account (i) and (11), we have
\begin{align}
\langle F_{ij}(a),F_{kl}(b)\rangle &=\sum_{t,s}\mu_t\overline{\lambda_s}\langle F_{ij}(\pi_t),F_{kl}(\pi_s)\rangle
=\sum_{t}\mu_t\overline{\lambda_t}\langle F_{ij}(\pi_t),F_{kl}(\textbf{1})\rangle\nonumber\\
&=\langle F_{ij}(\sum_{t}\mu_t\overline{\lambda_t}\pi_t),F_{kl}(\textbf{1})\rangle
=\langle F_{ij}(b^*a),F_{kl}(\textbf{1})\rangle.\nonumber\end{align}

\medskip (iv). We show first that
$$
\langle F_{12}(a\pi)+F_{21}(a^*\pi),F_{11}(\pi)\rangle=\langle F_{22}(\pi),F_{12}(a\pi)+F_{21}(a^*\pi)\rangle,
\quad a\in\mathcal{M},\ \pi\in\mathcal{M}^{\pr}.\eqno(12)$$
Define  $f= F_{12}(a\pi)+F_{21}(a^*\pi)=F{\small\left(\begin{array}{cc} 0 & a\pi\\
a^*\pi & 0 \end{array}\right)}$.  By virtue of (iii) and (7)
\begin{align}
\langle f&,F_{11}(\pi)\rangle =\langle f,F(\textbf{I} )\rangle -[ \langle f,F_{22}(\pi)\rangle +
\langle F_{22}(\pi),f\rangle]+\langle F_{22}(\pi),f\rangle \nonumber\\
&=\langle f,F(\textbf{I} )\rangle -2\RE{\small\left\langle F\left(\begin{array}{cc} 0 & a\pi\\
a^*\pi & 0 \end{array}\right), F\left(\begin{array}{cc} 0 & 0\\
 0 & \pi \end{array}\right)\right\rangle}+\langle F_{22}(\pi),f\rangle\nonumber\\
& =\varrho(a\pi\varepsilon_{12}+a^*\pi\varepsilon_{21})-\varrho((a\pi\varepsilon_{12}+a^*\pi\varepsilon_{21})
\cdot\pi\varepsilon_{22}+\pi\varepsilon_{22}\cdot(a\pi\varepsilon_{12}+a^*\pi\varepsilon_{21}))\nonumber\\
&+\langle F_{22}(\pi),f\rangle =\langle F_{22}(\pi),f\rangle,\nonumber\end{align}
and (12) is established. By setting  $a=\textbf{1}$ and $a=\im\!\textbf{1}$ in  (12), we have
$$
\langle F_{12}(\pi)+F_{21}(\pi),F_{11}(\pi)\rangle=\langle F_{22}(\pi),F_{12}(\pi)+F_{21}(\pi)\rangle,
\eqno(13)$$
$$
\im\langle F_{12}(\pi)-F_{21}(\pi),F_{11}(\pi)\rangle=-\im\langle F_{22}(\pi),F_{12}(\pi)-F_{21}(\pi)\rangle.
\eqno(14)$$
Multiplying both sides of  (14) by $\im$ and taking into account (13), we have (iv). This proves
the proposition.
\end{proof}

\begin{prop}
Let $F_{ij}\ (i=1,2)$ be OVFs
satisfying conditions  $(iii)$ and  $(iv)$ of Prop. 1.  Then equality  $(12)$ holds.
\end{prop}

\begin{proof}
Because the linear mappings  $F_{ij}$ are bounded, it is sufficient to prove
(12) for finite sums in the form  $a=\sum\limits_{s}\lambda_s\pi_s$ where $\lambda_s\in\mathbb{C}, \pi_s\in
\mathcal{M}^{\pr},\ \pi_s\pi_t=0\ (s\ne t), \sum\limits_{s}\pi_s=\textbf{1}$. In view of  (iii) and (iv) we
have
\begin{align}
\langle F_{12}(a\pi)+F_{21}(a^*\pi)&,F_{11}(\pi)\rangle=\sum_{s}[\lambda_s
\langle F_{12}(\pi\pi_s), F_{11}(\textbf{1})\rangle +\overline{\lambda_s}\langle F_{21}(\pi\pi_s),
F_{11}(\textbf{1})\rangle]\nonumber\\
&=\sum_{s}[\lambda_s \langle F_{22}(\textbf{1}), F_{21}(\pi\pi_s)\rangle +\overline{\lambda_s}\langle F_{22}(\textbf{1}),
F_{12}(\pi\pi_s)\rangle]\nonumber\\
&=\langle F_{22}(\pi), F_{21}(\sum_{s}\overline{\lambda_s} \pi\pi_s)\rangle+
\langle F_{22}(\pi), F_{12}(\sum_{s}\lambda_s \pi\pi_s)\rangle\nonumber\\
&=\langle F_{22}(\pi), F_{21}(a^*\pi)+F_{12}(a\pi)\rangle.\nonumber\end{align}
The proposition follows.
\end{proof}

\begin{theorem}
Let $\mathcal{N}=\mathcal{M}\otimes M_2$ be a $W^*$-algebra of type $I_2$,
$F_{ij}: \mathcal{M}\rightarrow H$ be $H$-valued OVFs ($w^*$-continuous OVFs) with the properties  $(i) - (iv)$
of Prop. 1. Then the linear mapping $F:\mathcal{N}\rightarrow H$ defined by the equalities
$$
F(a\varepsilon_{ij})\equiv F_{ij}(a),\quad a\in\mathcal{M},\ i,j=1,2,$$
is the OVF (respectively,  $w^*$-continuous OVF).
\end{theorem}

\begin{proof}
 Since  $F$ is  $w^*$--continuous if and only if  $F_{ij}$ are $w^*$-continuous,
it is sufficient to  consider the case when  $F_{ij}$ be  OVFs.
We have only to verify the equality
$$
pq=0\ \Rightarrow\ \langle F(p),F(q)\rangle=0,\quad p,q\in\mathcal{N}^{\pr}.$$
We use the canonical representation projections in  $\mathcal{N}$.
Specifically, every projection $r\in\mathcal{N}^{\pr}$ can be expressed  in the  form
[8, Lemma 1]
$$
r= \pi_1\oplus\pi_2+p(a,v,\pi),\eqno(15)$$
where
$$
 \pi_1\oplus\pi_2\equiv\left(\begin{array}{cc} \pi_1 & 0\\
0 & \pi_2 \end{array}\right),\quad \pi_i\in\mathcal{M}^{\pr},$$
$$
p(a,v,\pi)\equiv\left(\begin{array}{cc} a\pi & v\pi(a(\textbf{1}-a))^{1/2}\\
v^*\pi(a(\textbf{1}-a))^{1/2} & (\textbf{1}-a)\pi \end{array}\right),\quad \pi\in\mathcal{M}^{\pr},
v\in\mathcal{M}^{\un}.$$
In  addition, $\pi_i\leq\textbf{1}-\pi\ (i=1,2),\ a\in\mathcal{M},\ 0\leq a\leq\textbf{1}, \rp a(\textbf{1}-a))
=\textbf{1}$. Now we present  $p,q\in\mathcal{N}^{\pr}$ in the form of (15):
$$
p=\tau_1\oplus\tau_2+p(a,v,\tau_3),\quad q=\sigma_1\oplus\sigma_2+p(b,w,\sigma_3),$$
where $\tau_i,\sigma_i\in\mathcal{M}^{\pr},\ 0\leq a,b\leq\textbf{1},\ v,w\in\mathcal{M}^{\un}$.
From $pq=0$ it follows [8, Lemma 2] that
$$
\tau_i\sigma_i=\tau_i\tau_3=\sigma_i\sigma_3=\tau_3\sigma_i=\tau_i\sigma_3=0,\
w\pi=-v\pi,\ b\pi=(\textbf{1}-a)\pi, \eqno(16)$$
where  $i=1,2,\ \pi\equiv\sigma_3\tau_3$. Denoting for brevity $\varkappa(x)\equiv(x(\textbf{1}-x))^{1/2}$ we have
\begin{align}
F(p)&=F_{11}(\tau_1+a\tau_3)+ F_{12}(v\varkappa(a)\tau_3) + F_{21}(v^*\varkappa(a)\tau_3)+
 F_{22}(\tau_2+(\textbf{1}-a)\tau_3),\nonumber\\
F(q)&=F_{11}(\sigma_1+b\sigma_3)+ F_{12}(w\varkappa(b)\sigma_3) + F_{21}(w^*\varkappa(b)\sigma_3)+
 F_{22}(\sigma_2+(\textbf{1}-b)\sigma_3).\nonumber\end{align}
Taking the scalar product of vectors $F(p)$ and $F(q)$ with regard to (i), (iii), (iv) in Prop. 1 and
equalities (16) we have
\begin{align}
\langle F(p), F(q)\rangle&=\langle F_{11}(a\pi), F_{11}((\textbf{1}-a)\pi)\rangle -
\langle F_{11}(a\pi), F_{12}(v\varkappa(a)\pi)\rangle\nonumber\\
&-\langle F_{11}(a\pi), F_{21}(v^*\varkappa(a)\pi)\rangle + \langle F_{12}(v\varkappa(a)\pi),
F_{11}((\textbf{1}-a)\pi)\rangle\nonumber\\
&-\|F_{12}(v\varkappa(a)\pi)\|^2 +\langle F_{12}(v\varkappa(a)\pi),
F_{22}(a\pi)\rangle\nonumber\\
&+\langle F_{21}(v^*\varkappa(a)\pi),  F_{11}((\textbf{1}-a)\pi)\rangle- \|F_{21}(v^*\varkappa(a)\pi)\|^2\nonumber\\
&+\langle F_{21}(v^*\varkappa(a)\pi),  F_{22}(a\pi),\rangle -
\langle F_{22}((\textbf{1}-a)\pi), F_{12}(v\varkappa(a)\pi)\rangle\nonumber\\
&-\langle F_{22}((\textbf{1}-a)\pi), F_{21}(v^*\varkappa(a)\pi)\rangle+
\langle F_{22}((\textbf{1}-a)\pi), F_{22}(a\pi)\rangle.\nonumber\end{align}
By Prop. 1(ii), the first and the last summands on the right hand-side of obtained equality
are mutually annihilated with the fifth and the eights ones. Grouping the remained summands and setting
$c=va\varkappa(a),\ d= v(\textbf{1}-a)\varkappa(a)$ we have by  (12)
\begin{align}
\langle F(p)&, F(q)\rangle=[\langle F_{12}(v\varkappa(a)\pi)+ F_{21}(v^*\varkappa(a)\pi), F_{22}(a\pi)\rangle\nonumber\\
&- \langle F_{11}(a\pi), F_{12}(v\varkappa(a)\pi)+ F_{21}(v^*\varkappa(a)\pi)\rangle]\nonumber\\
&+[\langle F_{12}(v\varkappa(a)\pi)+ F_{21}(v^*\varkappa(a)\pi), F_{11}((\textbf{1}-a)\pi)\rangle \nonumber\\
&-\langle F_{22}((\textbf{1}-a)\pi), F_{12}(v\varkappa(a)\pi)+ F_{21}(v^*\varkappa(a)\pi)\rangle]\nonumber\\
&=[\langle F_{12}(c\pi)+ F_{21}(c^*\pi), F_{22}(\pi)\rangle -
\langle F_{11}(\pi), F_{12}(c\pi)+ F_{21}(c^*\pi)\rangle]\nonumber\\
&+[\langle F_{12}(d\pi)+ F_{21}(d^*\pi), F_{11}(\pi)\rangle -
\langle F_{22}(\pi), F_{12}(d\pi)+ F_{21}(d^*\pi)\rangle]=0.\nonumber\end{align}
The proof is complete.
\end{proof}

\centerline{\textbf{3. The stationarity of  OVFs over  ${\bf W^*}$-algebras of type ${\bf I_2}$}}

\bigskip As an application of Theorem 3 we will show that every $w^*$-continuous OVF over an $W^*$-algebra of type $I_2$
is stationary. Let  $\mathcal{N} =\mathcal{M}\otimes M_2$ be a $W^*$-algebra of type $I_2$, $F:\mathcal{N}\rightarrow H$
be a $w^*$-continuous OVF  and the functional $\varrho$ be given by (5). Let us observe some properties
of  reductions (10).  With the above notations of Section 1 we have
$$
\varrho(\pi\varepsilon_{ii})=\langle F(\pi\varepsilon_{ii}), F(\textbf{I})\rangle=
\langle F(\pi\varepsilon_{ii}), F(\varepsilon_{ii})\rangle
=\langle F_{ii}(\pi), F_{ii}(\textbf{1})\rangle,\eqno(17)$$
where  $i=1,2,\ \pi\in\mathcal{M}^{\pr}$. Therefore,
$$
\rho_{ii}=\frac{d}{d\nu}\langle F_{ii}(\cdot), F_{ii}(\textbf{1})\rangle,\quad i=1,2.\eqno(18)$$
Similarly,
$$
\rho_{ij}=\frac{d}{d\nu}\langle F_{ji}(\cdot), F_{11}(\textbf{1})+F_{22}(\textbf{1})\rangle,\quad i,j=1,2\
(i\ne j).\eqno(19)$$
In addition there are defined functions  $r_{ij}\in L^1(\Omega,\nu)^+$ such that
$$
\langle F_{ij}(\pi), F_{ij}(\textbf{1})\rangle=\int\limits_{\pi}r_{ij}\,d\nu,
\quad r_{ij}=\frac{d}{d\nu}\langle F_{ij}(\cdot), F_{ij}(\textbf{1})\rangle,\quad i,j=1,2.\eqno(20)$$
In this connection
$$
r_{ii}=\varrho_{ii}\ (i=1,2),\quad r_{12}+r_{21}= \varrho_{11}+\varrho_{22}\ \text{ a. e.}\eqno(21)$$
(The second equality in (21) follows from Prop. 1(ii).)

If  $F$ is stationary there exist (see (4))  $\varphi,\psi\in\mathcal{N}^+_*$ such that
$$
\langle F(x),F(y)\rangle= \varphi(y^*x)+\psi(xy^*),\quad x,y\in\mathcal{N}.\eqno(22)$$
In this case (8) is satisfied. Let $(\varphi_{ij})$ be the matrix of  $\varphi$
(see (2)). We have
\begin{align}
\int\limits_{\pi}\varphi_{21}\,d\nu&=\varphi(\pi\varepsilon_{12})=
\varphi(\varepsilon_{11}(\pi\varepsilon_{12})\varepsilon_{22})=\langle F((\pi\varepsilon_{12})\varepsilon_{22}),
F(\varepsilon_{11})\rangle\nonumber\\
&=\langle F(\pi\varepsilon_{12}),F(\varepsilon_{11})\rangle=\langle F_{12}(\pi),F_{11}(\textbf{1})\rangle,\quad
\pi\in\mathcal{M}^{\pr}.\nonumber\end{align}
Thus,
$$
\varphi_{21}=\frac{d}{d\nu}\langle F_{12}(\cdot),F_{11}(\textbf{1})\rangle,\quad
\varphi_{12}=\overline{\varphi_{21}}=\frac{d}{d\nu}\langle F_{21}(\cdot),F_{22}(\textbf{1})\rangle.$$
(The equality $\varphi_{12}=\overline{\varphi_{21}}$ follows from the positivity of  $\varphi$.)
Further,
\begin{align}
\int\limits_{\pi}(\varphi_{11}-\varphi_{22})\,d\nu&=\varphi(\pi\varepsilon_{11}-\pi\varepsilon_{22})=
\varphi(\pi\varepsilon_{12}\varepsilon_{12}^*)-\varphi(\pi\varepsilon_{12}^*\varepsilon_{12})\nonumber\\
&=\varphi(\pi\varepsilon_{12}\varepsilon_{12}^*)+\psi(\pi\varepsilon_{12}^*\varepsilon_{12})-
\varrho(\pi\varepsilon_{12}^*\varepsilon_{12})\nonumber\\
&=\langle F(\pi\varepsilon_{12}^*),F(\pi\varepsilon_{12}^*)\rangle-\varrho(\pi\varepsilon_{22})
\nonumber\\
&=\|F(\pi\varepsilon_{21})\|^2-\varrho(\pi\varepsilon_{22})=\langle F_{21}(\pi),F_{21}(\textbf{1})\rangle
-\int\limits_{\pi}\varrho_{22}\,d\nu\nonumber\\
&=\int\limits_{\pi}(r_{21}-\varrho_{22})\,d\nu.\nonumber\end{align}
(The function  $r_{21}$ is defined by (20).) Putting $\phi\equiv\varphi_{11}$ we have
$$
\varphi=\left(\begin{array}{cc}\phi & \varphi_{12}\\
\overline{\varphi_{12}} & \phi+\varrho_{22}-r_{21}\end{array}\right),\qquad
\psi=\left(\begin{array}{cc}\varrho_{11}-\phi & \varrho_{12}-\varphi_{12}\\
\overline{\varrho_{12}}-\overline{\varphi_{12}} &r_{21}- \phi\end{array}\right).\eqno(23)$$
It  should be noted  that all elements of matrices $\varphi$ and $\psi$, with the exception of $\phi$,
are defined as  the Radon-Nikodym derivatives associated with the fields  $F_{ij}$.
The function  $\phi\in L^1(\Omega,\nu)^+$  necessarily satisfies inequalities
$$
\max\{0,r_{21}(\omega)-\varrho_{22}(\omega)\}\leq\phi(\omega)\leq\min\{\varrho_{11}(\omega),r_{21}(\omega)\}\quad
\text{a.~e.},\eqno(24)$$
$$
|\varphi_{12}(\omega)|\leq[\phi(\omega)(\phi(\omega)+\varrho_{22}(\omega)-r_{21}(\omega))]^{1/2}\quad
\text{a.~e.},\eqno(25)$$
$$
|\varrho_{12}(\omega)-\varphi_{12}(\omega)|\leq[(\rho_{11}(\omega)-\phi(\omega))(r_{21}(\omega)-\phi(\omega))]^{1/2}
\quad \text{a.~e.}\eqno(26)$$

\begin{theorem}
Let $\mathcal{N} =\mathcal{M}\otimes M_2$ be a $W^*$-algebra of type $I_2$.
Then every  $w^*$-continuous OVF $F:\mathcal{N}\rightarrow H$ is stationary.
\end{theorem}

\medskip The following lemma gives conditionally ahead of time proof of the theorem.

\begin{lemma}
If the function  $\phi\in L^1(\Omega,\nu)$ satisfies inequalities (24) -- (26),
then the equality  (22) holds when $\varphi$ and  $\psi$ are defined by matrices (23).
\end{lemma}

\begin{proof}
As $F$ is linear and  $w^*$-continuous (and by Prop. 3(iii)) we must verify only
that the following equalities hold
$$
\langle F(\pi\varepsilon_{ij}), F(\varepsilon_{kl})\rangle=\varphi(\varepsilon_{lk}(\pi\varepsilon_{ij}))+
\psi((\pi\varepsilon_{ij})\varepsilon_{lk}),\ \pi\in\mathcal{M}^{\pr},\ i,j,k,l\in\{1,2\}.$$
They are  easily verified directly. For example, in view of  (17),(8),(9) and by Prop. 1(iv) we have
\begin{align}
\langle F(\pi\varepsilon_{ii}), F(\varepsilon_{ii})\rangle&=\int\limits_{\pi}\varrho_{ii}\,d\nu=
\int\limits_{\pi}(\varphi_{ii}+\psi_{ii})\,d\nu=\varphi(\pi\varepsilon_{ii})+\psi(\pi\varepsilon_{ii})\nonumber\\
&=\varphi(\varepsilon_{ii}(\pi\varepsilon_{ii}))+\psi((\pi\varepsilon_{ii})\varepsilon_{ii}),\quad i=1,2,\nonumber\\
\langle F(\pi\varepsilon_{11}), F(\varepsilon_{12})\rangle&=\langle F_{11}(\pi), F_{12}(\textbf{1})\rangle
=\langle F_{11}(\textbf{1}), F_{12}(\pi)\rangle\nonumber\\
&=\langle F_{21}(\pi), F_{22}(\textbf{1})\rangle=\int\limits_{\pi}\varphi_{12}\,d\nu=\varphi(\pi\varepsilon_{21})
\nonumber\\
&=\varphi((\pi\varepsilon_{21})\varepsilon_{11})+\psi(\varepsilon_{11}(\pi\varepsilon_{21})).\nonumber\end{align}
Similarly we can  verify remaining equalities, and the lemma follows.
\end{proof}

\begin{lemma}
Any OVF over the factor of type $I_2$ is stationary.
\end{lemma}

\begin{proof}
Let $F: M_2\rightarrow H$ be a OVF. Without loss of generality, we may  assume that
$\|F(\textbf{I})\|=1$. Let us denote for brevity $F_{ij}=F(\varepsilon_{ij})$.
By virtue of standard arguments connected with unitary invariance, it is sufficient to consider the case
when the matrix  $\varrho=(\varrho_{ij})$ (here it is scalar) is diagonal:
$$
\varrho= \left(\begin{array}{cc}\varrho_{11} & 0\\
0 & \varrho_{22}\end{array}\right),\quad \varrho_{11}+\varrho_{22}=1,\ \varrho_{ii}\geq0.$$
Also note that
$$
r_{ij}=\|F_{ij}\|^2,\quad r_{12}+r_{21}=1.$$
\textit{Case 1: the rank of matrix  $\varrho$ equals 1.} Let $\varrho_{11} = 1, \varrho_{22}=0$
(the same arguments can be used if $\varrho_{11} = 0, \varrho_{22}=1$). By the equality
$\varrho_{12}=0$ it follows that  $\{F_{ij}\}$ is the orthogonal system of vectors in  $H$.
Now we see that  equality  (22) is satisfied with
$$
\varphi= \left(\begin{array}{cc} r_{21} & 0\\
0 & 0\end{array}\right),\qquad \psi= \left(\begin{array}{cc} r_{12} & 0\\
0 & 0\end{array}\right)$$
In fact, for $x=(x_{ij}),\ y=(y_{ij})\in M_2$ we have
\begin{align}
\langle F(x), F(y)\rangle&=\sum_{i,j=1}^2 x_{ij}\overline{y_{ij}}\|F_{ij}\|^2
=x_{11}\overline{y_{11}}+r_{12}x_{12}\overline{y_{12}}+ r_{21}x_{21}\overline{y_{21}}\nonumber\\
&=r_{21}(x_{11}\overline{y_{11}}+x_{21}\overline{y_{21}})+ r_{12}(x_{11}\overline{y_{11}}+x_{12}\overline{y_{12}})
\nonumber\\
&=\tr\left[\left(\begin{array}{cc} r_{21} & 0\\
0 & 0\end{array}\right)\left(\begin{array}{cc} \overline{y_{11}} & \overline{y_{21}}\\
\overline{y_{12}} & \overline{y_{22}}\end{array}\right)\left(\begin{array}{cc} x_{11} &  x_{12}\\
 x_{21} &  x_{22}\end{array}\right)\right]\nonumber\\
 &+ \tr\left[\left(\begin{array}{cc} r_{12} & 0\\
0 & 0\end{array}\right)\left(\begin{array}{cc} x_{11} &  x_{12}\\
 x_{21} &  x_{22}\end{array}\right)\left(\begin{array}{cc} \overline{y_{11}} & \overline{y_{21}}\\
\overline{y_{12}} & \overline{y_{22}}\end{array}\right)\right]\nonumber\\
&=\varphi(y^*x)+\psi(xy^*).\nonumber\end{align}
It is easily seen that $\varphi$ and $\psi$  are determined uniquely.

\medskip\textit{Case 2: the rank of matrix  $\varrho$ equals 2.}\ In the case
$$
\varrho= \left(\begin{array}{cc}\varrho_{11} & 0\\
0 & \varrho_{22}\end{array}\right),\quad \varrho_{11}+\varrho_{22}=1,\ \varrho_{ii}>0,$$
and the required matrices $\varphi$ and  $\psi$ obtain the form
$$
\varphi= \left(\begin{array}{cc}\phi & \varphi_{12}\\
\overline{\varphi_{12}} & \phi+\varrho_{22}-r_{21}\end{array}\right),\quad
\psi= \left(\begin{array}{cc}\varrho_{11}-\phi & -\varphi_{12}\\
-\overline{\varphi_{12}} & r_{21}-\phi\end{array}\right),$$
where
$$
\varphi_{12}=\langle F_{21}, F_{22}\rangle=-\langle F_{21}, F_{11}\rangle=\langle F_{11}, F_{12}\rangle
\eqno(27)$$
(the second equality follows from $\varrho_{12}=0$ (see (22)), the third one does from Prop. 1(iv), and
$\phi$ is an unknown for now parameter   satisfying inequalities (see (24) -- (26))
$$
\max\{0,r_{21}-\varrho_{22}\}\leq\phi\leq\min\{\varrho_{11},r_{21}\},\eqno(28)$$
$$
|\varphi_{12}|^2\leq \phi(\phi+\varrho_{22}-r_{21}),\eqno(29)$$
$$
|\varphi_{12}|^2\leq(\rho_{11}-\phi)(r_{21}-\phi).\eqno(30)$$
We show that  $\phi_0=\frac12[r_{21}-\varrho_{22}+\sqrt{(r_{21}-\varrho_{22})^2+4|\varphi_{12}|^2}]$,
a solution of the quadratic equation  $\phi(\phi+\varrho_{22}-r_{21})- |\varphi_{12}|^2=0$,
satisfies  inequalities  (28) -- (30). It is evident that $\phi_0$ satisfies (29) and $\phi_0\geq0$.
By  (27)
$$
2\phi_0\leq r_{21}-\varrho_{22}+\sqrt{(r_{21}-\varrho_{22})^2+4r_{21}\varrho_{22}}=2r_{21},$$
\begin{align}
2\phi_0&\leq \varrho_{11}-r_{12}+\sqrt{(r_{12}-\varrho_{11})^2+4|\varphi_{12}|^2}\leq
\varrho_{11}-r_{12}+\sqrt{(r_{12}-\varrho_{11})^2+4r_{12}\varrho_{11}}\nonumber\\
&=2\varrho_{11}.\nonumber\end{align}
Hence we have $\phi_0\leq\min\{\varrho_{11},r_{21}\}$ and (28) follows.

\medskip In order to verify  (30) we  represent  $F$ in the coordinate form. As $\dim\rg F\leq4$
we will assume that  $F$ acts into  $\mathbb{C}^4$.
By a suitable choice of the coordinate system in $\mathbb{C}^4$ we may assume that
$$
F(I)=(1,0,0,0),\quad F_{11}=(\alpha,\zeta,0,0),\quad F_{22}=(1-\alpha,-\zeta,0,0).$$
In this case
$$
0<\alpha<1,\quad |\zeta|^2=\alpha-\alpha^2,\  \zeta=\omega|\zeta|,\ |\omega|=1.$$
Put  $F_{12}=(\xi_1,\xi_2,\xi_3,\xi_4),\ F_{21}=(\eta_1,\eta_2,\eta_3,\eta_4)$ and note that
$$
\eta_1=\langle F(\varepsilon_{21}), F(I)\rangle=\tr(\varrho\varepsilon_{21})=0,$$
Similarly,  $\xi_1=0$. In addition,  $\eta_2\overline{\zeta}=-\overline{\xi_2}\zeta$ (by Prop. 1(iv)) so that
 $\eta_2=-\omega^2\overline{\xi_2}$. Thus, putting $\xi\equiv\xi_2$,
$$
F_{12}=(0,\xi,\xi_3,\xi_4),\quad F_{21}=(0,-\omega^2\overline{\xi},\eta_3,\eta_4),$$
and also according to Prop. 1(i),(ii),
$$
2|\xi|^2+|\xi_3|^2+|\xi_4|^2+|\eta_3|^2+|\eta_4|^2=1,\quad \xi_3\overline{\eta_3}+\xi_4\overline{\eta_4}
=\overline{\omega}^2\xi^2.$$
It should be noted that
\begin{align}
r_{12}r_{21}&=(|\xi|^2+|\xi_3|^2+|\xi_4|^2)(|\xi|^2+|\eta_3|^2+|\eta_4|^2)\nonumber\\
&=|\xi|^2(1-|\xi|^2)+(|\xi_3|^2+|\xi_4|^2)(|\eta_3|^2+|\eta_4|^2)\nonumber\\
&\geq|\xi|^2-|\xi|^4+|\xi_3\overline{\eta_3}+\xi_4\overline{\eta_4}|^2=|\xi|^2.\nonumber\end{align}
Therefore,
$$
|\varphi_{12}|^2=|\langle F_{21}, F_{11}\rangle|^2=|\xi|^2|\zeta|^2\leq r_{12}r_{21}\varrho_{11}\varrho_{22}.$$
Since $r_{12}+r_{21}= \varrho_{11}+\varrho_{22}=1$, we have
$$
\varrho_{22}-r_{21}=\varrho_{22}r_{12}-\varrho_{11}r_{21}.$$
Now (30) follows from computation:
\begin{align}
2[(\varrho_{11}-\phi_0)(r_{21}&-\phi_0)-|\varphi_{12}|^2]=2(\varrho_{11}r_{22}-\phi_0)\nonumber\\
&=2\varrho_{11}r_{21}-r_{21}+\varrho_{22}-\sqrt{(\varrho_{22}-r_{21})^2+4|\varphi_{12}|^2}\nonumber\\
&=\varrho_{11}r_{21}+\varrho_{22}r_{12}-\sqrt{(\varrho_{11}r_{21}-\varrho_{22}r_{12})^2+4|\varphi_{12}|^2}\nonumber\\
&\geq\varrho_{11}r_{21}+\varrho_{22}r_{12}-\sqrt{(\varrho_{11}r_{21}-\varrho_{22}r_{12})^2+
4r_{12}r_{21}\varrho_{11}\varrho_{22}}=0.\nonumber\end{align}
Applying Lemma 5 to the case in question, we finish the proof.
\end{proof}

\medskip \noindent\textit{Proof of Theorem 4. } in view of lemma 5 it is sufficient to show the existence of
a function $\phi\in L^1(\Omega,\nu)$ with the properties set out in (24) -- (26).
It is convenient to begin with the atomic case. Let $\Delta=\{\pi_k\}_{k\in K}$
be an arbitrary measurable partition of  $\Omega$, i. e.
$$
\Omega=\bigcup\limits_k \pi_k,\quad \pi_k\cap\pi_l=\emptyset\ (k\ne l).$$
Via  $\mathcal{M}_{\Delta}$ we denote the algebra of functions  $x\in L^{\infty}(\Omega,\nu)$ in the form of
 $x=\sum\limits_k\lambda_k\pi_k,\ \lambda_k\in\mathbb{C},\ \pi_k\in\Delta$.
Let $\mathcal{N}_{\Delta}=\mathcal{M}_{\Delta}\otimes M_2$ be a subalgebra of $\mathcal{N}$,
being the direct sum of factors of type $I_2$:
$$
\mathcal{N}_{\Delta}=\sum_k \pi_k\otimes M_2\quad \text{ where }\pi\otimes M_2\equiv\left\{
\left(\begin{array}{cc}\lambda_{11}\pi & \lambda_{12}\pi\\
\lambda_{21}\pi & \lambda_{22}\pi\end{array}\right):\ \lambda_{ij}\in\mathbb{C}\right\},$$
Particularly $\mathcal{N}_{\Delta}$ is a $W^*$-algebra.  The restriction  $F\,|\,\mathcal{N}_{\Delta}$ is
a $w^*$-continuous OVF over $\mathcal{N}_{\Delta}$. In this case the parametric functions characterizing
$F\,|\,\mathcal{N}_{\Delta}$ are defined as follows
$$
\varrho_{ii}^{\Delta}=\sum_k \varrho_{iik}^{\Delta}\pi_k,\quad \varrho_{iik}^{\Delta}
=\langle F_{ii}(\pi_k),F_{ii}(\textbf{1})\rangle,$$
$$
\varrho_{ij}^{\Delta}=\sum_k \varrho_{ijk}^{\Delta}\pi_k,\quad \varrho_{ijk}^{\Delta}
=\langle F_{ji}(\pi_k),F_{11}(\textbf{1})+F_{22}(\textbf{1})\rangle,\ i\ne j,$$
$$
r_{ij}^{\Delta}=\sum_k r_{ijk}^{\Delta}\pi_k,\quad r_{ijk}^{\Delta}=\|F_{ji}(\pi_k)\|^2,$$
$$
\varphi_{12}^{\Delta}=\overline{\varphi_{21}^{\Delta}}=\sum_k \varphi_{12k}^{\Delta}\pi_k,\quad
\varphi_{12k}^{\Delta}=\langle F_{21}(\pi_k),F_{11}(\textbf{1})\rangle,$$
with
$$
r_{12}^{\Delta}+r_{21}^{\Delta}=\varrho_{11}^{\Delta}+\varrho_{22}^{\Delta}.$$
According to Lemma 6 and  (18) -- (20) it is defined a function $\phi^{\Delta}\in L^1(\Omega,\nu)$ such that
$$
\max\{0,r_{21}^{\Delta}(\omega)-\varrho_{22}^{\Delta}(\omega)\}\leq\phi^{\Delta}(\omega)\leq
\min\{\varrho_{11}^{\Delta}(\omega),r_{21}^{\Delta}(\omega)\}\text{  a. e.},\eqno(31)$$
$$
|\varphi_{12}^{\Delta}(\omega)|\leq[ \phi^{\Delta}(\omega)(\phi^{\Delta}(\omega)+
\varrho_{22}^{\Delta}(\omega)-r_{21}^{\Delta}(\omega))]^{1/2}\text{  a. e.},$$
$$
|\varrho_{12}^{\Delta}(\omega)-\varphi_{12}^{\Delta}(\omega)|\leq[\rho_{11}^{\Delta}(\omega)-
\phi^{\Delta}(\omega)]^{1/2}[r_{21}^{\Delta}(\omega)-\phi^{\Delta}(\omega)]^{1/2}\text{  a. e.}$$

 Let us turn to the general case. Since $(\Omega, \nu)$ is a localizable measure space,
for the construction of a function $\phi\in L^1(\Omega,\nu)^+$ we can restrict our attention to the case when
$\nu(\Omega)<+\infty$. We will choose a sequence  $\Delta_n=\{\pi_{kn}\}_k$ of partitions  $\Omega$ such that
$$
|\varrho_{11}(\omega)-\varrho^n_{11k}|,\ |r_{21}(\omega)-r_{21k}^n|,\ ||\varphi_{12}(\omega)|
-|\varphi_{12k}^n||<\frac1n,\quad \omega\in\pi_{kn}.$$
(Here, for example, $\varrho_{11k}^n\equiv\varrho_{11k}^{\Delta_n}$.) The partition $\{\pi_{kn}\}_k$ may be
constructed in the following way: denoting
$$
\sigma_{tn}=\{\omega: \frac tn\leq \varrho_{11}(\omega)<\frac{t+1}n\},$$
$$
\sigma'_{tn}=\{\omega: \frac tn\leq r_{21}(\omega)<\frac{t+1}n\},$$
$$
\sigma''_{tn}=\{\omega: \frac tn\leq|\varphi_{21}(\omega)|<\frac{t+1}n\},\quad t=0,1,2,...,$$
we put for multi-index $k=(t,s,m),\ t,s,m=0,1,2,...$
$$
\pi_{kn}=\sigma_{tn}\cap\sigma'_{sn}\cap\sigma''_{mn},\ \text{  if } \nu(\pi_{kn})>0.$$
As $k=(t,s,m)$ we have
$$
\frac tn\leq \varrho_{11}(\omega)<\frac{t+1}n,\quad \omega\in\pi_{kn} \Rightarrow$$
$$
\frac tn\nu(\pi_{kn})\leq \int\limits_{\pi_{kn}}\varrho_{11}\,d\nu=\varrho_{11k}^n\nu(\pi_{kn})\leq
\frac{t+1}n\nu(\pi_{kn})\Rightarrow$$
$$
\frac tn\leq \varrho_{11k}^n\leq\frac{t+1}n\quad\Rightarrow\quad |\varrho_{11}(\omega)-\varrho_{11k}^n|
<\frac 1n\quad (\omega\in\pi_{kn}).$$
The arbitrariness of $k$ and similar computations for $r_{21},\ |\varphi_{12}|$  then  implies
$$
|\varrho_{11}(\omega)-\varrho^n_{11k}|,\ |r_{21}(\omega)-r_{21k}^n|,\ ||\varphi_{12}(\omega)|
-|\varphi_{12k}^n||<\frac1n,\quad \text{a. e.}$$
Thus, $\varrho_{11}^n, r_{21}^n,\ |\varphi_{12}^n|$ are the sequences of simple integrable functions that
converge uniformly (respectively) to $\varrho_{11}, r_{21},\ |\varphi_{12}|$.
As a consequence,  $\varrho_{22}^n \rightrightarrows
\varrho_{22}, r_{12}^n\rightrightarrows r_{12}$. In accordance with Lemma 6 we put
$$
\phi^{\Delta_n}(\omega)=\frac12[r_{21}^n(\omega)-\varrho_{22}^n(\omega)+\sqrt{(r_{21}^n(\omega)-
\varrho_{22}^n(\omega))^2+4|\varphi_{12}^n(\omega)|^2}].$$
By (31) the sequence  $\phi^{\Delta_n}$ is dominated with the integrable function
$\min\{\varrho_{11},r_{21}\}$  (we keep in mind that $\nu(\Omega)<+\infty$).
So by  the Dominated Convergence Theorem, the function $\phi(\omega)\equiv \lim_n \phi^{\Delta_n}(\omega),
\ \omega\in\Omega$, is integrable and satisfies inequalities  (24) -- (26). The proof is complete.$\qed$


\begin{thebibliography}{9999}


\bibitem[1]{1}  Masani P.,
{\it   Orthogonally scattered measure},
Advances in Math. {\bf 2}\ (1968),  61--117.

\bibitem[2]{2}  Masani P., {\it Quasi-isometric measures and their applications},
Bull. Amer. Math. Soc., {\bf 76}\ (1976),  427--528.

\bibitem[3]{3} Goldstein S., Jaite R., \textit{Second-order fields over $W^*$-algebras}, Bull. de l'Acad.
polon. de sci., ser. math., {\bf 30}\ (1982), no. 5--6, 255--259.

\bibitem[4]{4}  Sherstnev A.~N.,
{\it Methods of bilinear forms in  non-commutative measure and integral theory},
M.: Fizmatlit, 2008, 259 p. (in Russian)

\bibitem[5]{5} Sakai S., {\it $C^*$-Algebras and $W^*$-Algebras},
N.Y.--Heidelberg--Berlin: Springer-Verlag, 1971, 256 p.

\bibitem[6]{6} Segal I., {\it  Equivalence of measure spaces},
Amer. J. Math., {\bf 73}\ (1951), 275-313.

\bibitem[7]{7} Lugovaya G.~D., Sherstnev A.~N.,\textit{On topological properties of orthogonal
vector fields}, Lobachevskii Journal of Mathematics, , Vol. {\bf 32} (2011), no. 2,  125-127.

\bibitem[8]{8} Sherstnev A.N., \textit{Measures on projections in a $W^*$-algebra
of type $I_2$}, arXiv.1112.5569v1 [math.OA] 23 Dec 2011, 8 p.

\end{thebibliography}
\end{document}